\documentclass[final]{siamltex}
\usepackage{lipsum}
\usepackage{amsfonts}
\usepackage{graphicx}
\usepackage{epstopdf}
\usepackage{mathrsfs}
\usepackage{amsmath}
\usepackage{amssymb}
\usepackage{algorithmic}
\usepackage{multirow}
\usepackage{xcolor}
\usepackage{subfig}
\usepackage[numbers,sort&compress]{natbib}
\usepackage{latexsym,array,extarrows}
\usepackage{float}
\usepackage{amsopn}
\usepackage{multirow}

\usepackage{scrtime}
\usepackage[dont-mess-around]{fnpct}

\usepackage{etoolbox}
\makeatletter
\patchcmd\maketitle{\setcounter{footnote}{0}}{}{}{}
\patchcmd\maketitle{%
	\renewcommand\thefootnote{\@fnsymbol\c@footnote}}{\AdaptNote\thanks\multthanks}{}{}
\patchcmd\maketitle{%
	\def\@makefnmark{\rlap{\@textsuperscript{\normalfont\@thefnmark}}}}{}{}{}
\makeatother

\ifpdf
\DeclareGraphicsExtensions{.eps,.pdf,.png,.jpg}
\else
\DeclareGraphicsExtensions{.eps}
\fi
\makeatletter
\def\@textbottom{\vskip \z@ \@plus 200pt}
\let\@texttop\relax
\makeatother
\numberwithin{theorem}{section}

\newtheorem{rem}{Remark}
\newtheorem{assump}{Assumption}
\newtheorem{expl}{Example}

\allowdisplaybreaks
\title{An All-at-Once Preconditioner for Evolutionary Partial Differential Equations\thanks{This research was supported by research Grants, 12200317, 12300218, 12300519, 17201020
		from HKRGC GRF and 11801479 from NSFC}}
		
\author{Xue-lei Lin\thanks{Shenzhen JL Computational Science and Applied Research Institute, Shenzhen, P.R. China., and 
Beijing Computational Science Research Center, Beijing 100193, China.}
	\and
	Michael K. Ng\thanks{Department of Mathematics, The University of Hong Kong}}

\begin{document}	
		\maketitle	
		\begin{abstract}
			In [McDonald, Pestana and Wathen, \textit{SIAM J. Sci. Comput.}, 40 (2018), pp. A1012--A1033], a block circulant preconditioner is proposed for all-at-once linear systems arising from evolutionary partial differential equations, in which the preconditioned matrix is proven to be diagonalizable and to have identity-plus-low-rank decomposition in the case of the heat equation. In this paper, we generalize the block circulant preconditioner by introducing a small parameter $\epsilon>0$ into the top-right block of the block circulant preconditioner. The implementation of the generalized preconditioner requires the same computational complexity as that of the block circulant one.
			Theoretically, we prove that (i) the generalization preserves the diagonalizability and the identity-plus-low-rank decomposition; (ii) all eigenvalues of the new preconditioned matrix are clustered at 1 for sufficiently small $\epsilon$; (iii) GMRES method for the preconditioned system has a linear convergence rate independent of size of the linear system when $\epsilon$ is taken to be smaller than or comparable to square root of time-step size. Numerical results are reported to confirm the efficiency of the proposed preconditioner and to show that the generalization improves the performance of block circulant preconditioner.
		\end{abstract}
		
	\section{Introduction}\label{introduction}
  In this paper, we are particularly interested in evolutionary partial differential equations (PDEs) with first order temporal derivative. Classical time-stepping method solve evolutionary PDEs one time step after one time step (i.e., in a fully sequential manner), which would be time-consuming if the number of time steps are large. This motivates the development of parallel-in-time (PinT) methods for evolutionary PDEs during the last two decades. Among these, we mention the parareal algorithm \cite{lions2001661} and a closely related algorithm  multigrid-reduction-in-time (MGRiT) algorithm \cite{falgout2014parallel}, which attract considerable attention in recent years. Convergence of parareal algorithm and MGRiT are respectively justified in \cite{gander2007analysis} and \cite{dobrev2017two}. Many efforts are devoted to improving these two PinT algorithms and in particular the authors \cite{wu2018toward} and \cite{wu2019acceleration} proposed a novel coarse grid correction, which shows great potential for increasing the speedup according to the numerical results in \cite{kwok2019schwarz}. There are also many another PinT algorithms with completely different mechanism from parareal algorithm and MGRiT, such as the space-time multigrid algorithms \cite{gander2016analysis,hackbusch1985parabolic,horton1995space} and the diagonalization-based all-at-once algorithms \cite{gander2016direct,mcdonald2018preconditioning,mcdonald2017preconditioning,wathen2019note,gander2019convergence}. For an overview, we refer the interested reader to \cite{gander201550}.
 
 Recently, McDonald, Pestana and Wathen in \cite{mcdonald2018preconditioning} proposed a block circulant preconditioner to accelerate the convergence of Krylov subspace methods for solving the all-at-once linear system arising from backward-difference time discretization of evolutionary PDEs. It is interesting that the preconditioned system in \cite{mcdonald2018preconditioning} is diagonalizable in the case of the heat equation although the original all-at-once system is not diagonalizable, which would be useful in aspects of theoretical
 convergence analysis.
Moreover, the preconditioned matrix in \cite{mcdonald2018preconditioning} has an identity-plus-low-rank decomposition, which is usually related to fast convergence of the GMRES method. However, in \cite{mcdonald2018preconditioning}, the convergence of GMRES for the preconditioned system has not been proven to be independent of spatial discretization step-size yet.
 
 In this paper, we generalize the block circulant preconditioner proposed in \cite{mcdonald2018preconditioning} by introducing a parameter $\epsilon>0$ into the top-right block of the block circulant preconditioner. We call the generalized preconditioner by block $\epsilon$-circulant (BEC) preconditioner (when $\epsilon=1$, the BEC preconditioner is identical to block circulant preconditioner).
 Theoretically, we show that (i) the generalization preserves the diagonalizability and the identity-plus-low-rank decomposition; (ii) all eigenvalues of the preconditioned matrix by BEC preconditioner are clustered at 1 for sufficiently small $\epsilon$; (iii) GMRES method (restarted or non-restarted) for the preconditioned system has a linear convergence rate independent of both temporal and spatial step-sizes when $\epsilon$ is taken to be smaller than or comparable to square root of the temporal-step size. 
 When using Krylov subspace methods to solve the preconditioned linear system, it requires to compute the inverse of the block $\epsilon$-circulant preconditioner multiplied with some given vectors. To compute the matrix-vector multiplication efficiently, we resort to the fact that the block $\epsilon$-circulant preconditioner is diagonalizable by means of fast Fourier transform (FFT) with each eigen-block having the same size as that of the spatial discretization matrix. That means, to compute the inverse of BEC preconditioner multiplying a given vector is equivalent to solving a block diagonal linear system in Fourier domain. If the spatial term is the Laplace operator and the uniform spatial grid is employed, then the diagonal blocks of the block diagonal linear system are further diagonalizable by fast sine transform (FST), due to which the computation of the inverse of the BEC preconditioner times a vector is fast and exact. If the spatial term consists of some more general differential operators, then we resort to some efficient iterative solvers (e.g., multigrid method) as inner spatial solver. The details of the implementation of the preconditioned matrix times a vector are given in Section \ref{implementation}, which shows that the total storage of the proposed implementation is proportional to the number of unknowns and the total computational cost of the proposed implementation is proportional to the number of unknowns multiplied with its logarithm. GMRES method is employed to solve the preconditioned linear system. Numerical results for heat equation and convection-dominated convection diffusion are reported to show that the BEC preconditioner is efficient and it improves the performance of block circulant preconditioner.
 
 The outline of this paper is organized as follows. In Section \ref{discretization}, the all-at-once linear system arising from an evolutionary PDE is presented. In Section \ref{preconditioning}, the BEC preconditioner is proposed, the properties of the preconditioned system and convergence of GMRES for the preconditioned system are analyzed. In Section \ref{implementation}, the implementation of the preconditioned matrix-vector multiplication and the complexity of GMRES method are discussed. In Section 5, Numerical results are reported. Finally, concluding remarks are given in Section 6.
 \section{The All-at-Once System for Evolutionary PDEs}\label{discretization}
 As in \cite{mcdonald2018preconditioning}, we start with the following heat equation to describe our method clearly:
 \begin{align}
 &\partial_t u({\bf x},t)=\nabla(a({\bf x})\nabla u({\bf x},t))+f({\bf x},t),\quad ({\bf x},t)\in\Omega\times(0,T],\quad \Omega\subset \mathbb{R}^{2}{\rm~or~}\mathbb{R}^{3},\label{heateq}\\
 &u({\bf x},t)=g({\bf x},t),\quad ({\bf x},t)\in\partial\Omega,\label{dtbc}\\
 &u({\bf x},0)=u_0({\bf x}),\quad {\bf x}\in\bar{\Omega},\label{inicd}
 \end{align}
 where $\Omega$ is open, $\partial\Omega$ denotes boundary of $\Omega$, $f$, $g$ and $u_0$ are all given functions, $a({\bf x})$ is a given positive function.
 
 For a positive integer $N$, denote $\tau=\frac{T}{N}$ and $t_n=n\tau$ for $n=0,1,...,N$. The backward difference scheme is employed to discretize $\partial_t$, i.e., we adopt the discretization:
 \begin{align}\label{backwarddiff}
 \partial_t u({\bf x},t_n)\approx \frac{u({\bf x},t_n)-u({\bf x},t_{n-1})}{\tau},\quad n=1,2,...,N
 \end{align}
 
 Let $J$ be a positive integer. Denote the mass matrix by ${\bf M}\in\mathbb{R}^{J\times J}$ and denote the discretization of $-\nabla(a({\bf x})\nabla \cdot)$ by ${\bf K}\in\mathbb{R}^{J\times J}$.
 
 Then, \eqref{heateq}--\eqref{inicd} is discretized as follows
 \begin{equation}\label{stepsystem}
 {\bf M}\left(\frac{{\bf u}^{n}-{\bf u}^{n-1}}{\tau}\right)+{\bf K}{\bf u}^{n}={\bf f}^{n},\quad n=1,2,...,N,
 \end{equation}
 where ${\bf f}^{n}$ ($n=1,2,...,N$) consists of discretization of $f$ and $g$, ${\bf u}^{0}$ is discretization of $u_0$ on the spatial mesh, the unknowns ${\bf u}^{n}$ $(n=1,2,...,N)$ are approximation of $u(\cdot,t_n)$ on the spatial mesh.
 
 For column vectors ${\bf v}_i$ ($i=1,2,...,m$), we use the notation $({\bf v}_1;{\bf v}_2;\cdots;{\bf v}_m)$ to denote the following column vector:
 \begin{equation*}
 \left[\begin{array}[c]{c}
 {\bf v}_1\\
 {\bf v}_2\\
 \vdots\\
 {\bf v}_m
 \end{array}\right].
 \end{equation*} 
 Putting the $N$ many linear systems into a large linear system, we obtain
 \begin{equation}\label{allatoncesystem}
 {\bf L}{\bf u}={\bf f},
 \end{equation}
 where
 \begin{align*}
 &{\bf u}=({\bf u}^{1};{\bf u}^{2};\cdots;{\bf u}^{N}),\quad {\bf f}=(\tau{\bf f}^{1}+{\bf M}{\bf u}^{0};\tau{\bf f}^{2};\tau{\bf f}^{3};\cdots;\tau{\bf f}^{N}),\\
 &{\bf L}=\left[
 \begin{array}
 [c]{cccc}
 {\bf A}_0& &   &  \\
 -{\bf M}& {\bf A}_0&  &\\
 ~&\ddots&\ddots&\\
 ~& ~& -{\bf M} & {\bf A}_0
 \end{array}
 \right]\in\mathbb{R}^{NJ\times NJ},\quad {\bf A}_0={\bf M}+\tau{\bf K}.
 \end{align*}
 
 \begin{rem}
 	The BEC preconditioner is still available when \eqref{backwarddiff} is replaced by multi-step backward difference schemes, the details of which are discussed in Section \ref{implementation}. For the purpose of analysis and fast implementation, some assumptions of ${\bf M}$ and ${\bf K}$ are listed as follows
 	\begin{assump}\label{spdassump}
 		Both ${\bf M}$ and ${\bf K}$ are real symmetric positive definite.
 	\end{assump}
 	\begin{assump}\label{regularmassassump}
 		The condition number, $\kappa_2({\bf M})$, of ${\bf M}$ is uniformly bounded, i.e., $\sup\limits_{J\in\mathbb{N}^{+}}\kappa_2({\bf M})<+\infty$.
 	\end{assump}
 	\begin{assump}\label{sparseassumpp}
 		Both ${\bf M}$ and ${\bf K}$ are sparse, i.e., ${\bf M}$ and ${\bf K}$ only have $\mathcal{O}(J)$ many nonzero entries.
 	\end{assump}
 
 	The Assumption \ref{spdassump} is fulfilled by a lot of discretization schemes, such as, central difference method, finite element methods. The Assumption \ref{regularmassassump} is quite obvious when the spatial discretization is of finite difference type, since in that case ${\bf M}$ is exactly an identity matrix. Moreover, Assumption \ref{regularmassassump} is also fulfilled for finite element discretization whenever the mesh is simplicial and quasi-uniform; see \cite{kamenski2014conditioning}. The Assumption \ref{sparseassumpp} is obvious for finite difference method or finite element methods with locally supported basis.
 \end{rem}
 
 \section{The BEC preconditioner and Analysis of the Preconditioned System by BEC Preconditioner}\label{preconditioning}
 In this section, we propose the BEC preconditioner and investigate some interesting properties such as identity-plus-low-rank decomposition, spectral clustering and diagonalizability of the preconditioned matrix. These properties may not be directly related to fast convergence of iterative solver for the preconditioned system (see, e.g., \cite{greenbaum1996any,arioli1998krylov}). In the end of this section, we will also prove that the GMRES method for the preconditioned system has a linear convergence rate independent of $N$ and $J$ when $\epsilon\lesssim\sqrt{\tau}$.
 
 The BEC preconditioner for the all-at-once system \eqref{allatoncesystem} is defined as
 \begin{align*}
 {\bf P}_{\epsilon}=\left[
 \begin{array}
 [c]{cccc}
 {\bf A}_0& &   & -\epsilon{\bf M} \\
 -{\bf M}& {\bf A}_0&  &\\
 ~&\ddots&\ddots&\\
 ~& ~& -{\bf M} & {\bf A}_0
 \end{array}
 \right]\in\mathbb{R}^{NJ\times NJ},
 \end{align*}
 where $\epsilon>0$ is a parameter.
 When $\epsilon=1$, then ${\bf P}_{\epsilon}$ is exactly the block circulant preconditioner proposed in \cite{mcdonald2018preconditioning}.
 
 It is clear that ${\bf L}$ is invertible. Moreover, since ${\bf L}$ is a block lower triangular Toeplitz matrix, ${\bf L}^{-1}$ is also a block lower triangular Toeplitz matrix, which can be rewritten as follows \cite{mcdonald2018preconditioning}
 \begin{equation}\label{elinvlttexpres}
 {\bf L}^{-1}=\left[
 \begin{array}
 [c]{cccc}
 ({\bf L}^{-1})_{0} &       &  &\\
 ({\bf L}^{-1})_{1}& ({\bf L}^{-1})_{0}&  &\\
 \vdots&\ddots &\ddots&\\
 ({\bf L}^{-1})_{N-1}& \ldots & ({\bf L}^{-1})_{1} & ({\bf L}^{-1})_{0}
 \end{array}
 \right],\qquad\begin{array}[c]{c}
 ({\bf L}^{-1})_{k}:=({\bf A}_0^{-1}{\bf M})^{k}{\bf A}_0^{-1}\\
 k=0,1,...,N-1.
 \end{array}
 \end{equation}
 Denote by ${\bf I}_k$, the $k\times k$ identity matrix. Let ${\bf e}_i$ be $i$th column of ${\bf I}_{N}$. Denote ${\bf E}_i={\bf e}_i\otimes{\bf I}_{J}$.
 
 For any Hermitian positive semi-definite matrix ${\bf H}\in\mathbb{C}^{m\times m}$, denote $${\bf H}^{\frac{1}{2}}:={\bf U}^{*}{\rm diag}(d_1^{\frac{1}{2}},d_2^{\frac{1}{2}},...,d_m^{\frac{1}{2}}){\bf U},$$
 where ${\bf U}^{*}{\rm diag}(d_1,d_2,...,d_m){\bf U}$ is unitary diagonalization of ${\bf H}$. In particular, if ${\bf H}$ is Hermitian positive definite, then we rewrite $({\bf H}^{-1})^{\frac{1}{2}}$ as ${\bf H}^{-\frac{1}{2}}$ for notation simplification.
 
 For any square matrix ${\bf C}$, denote by $\sigma({\bf C})$ the spectrum of ${\bf C}$.
 Denote
 \begin{equation}\label{zepsdef}
 {\bf Z}_{\epsilon}:=\epsilon^{-1}[{\bf I}_{J}-\epsilon({\bf A}_0^{-1}{\bf M})^{N}]{\bf M}^{-1}.
 \end{equation}
 \begin{theorem}\label{pinvertibility}
 	Let $\epsilon\in(0,1]$. Then, both ${\bf P}_{\epsilon}$ and ${\bf Z}_{\epsilon}$ are invertible with ${\bf P}^{-1}_{\epsilon}={\bf L}^{-1}+{\bf L}^{-1}{\bf E}_1{\bf Z}_{\epsilon}^{-1}{\bf E}_{N}^{\rm T}{\bf L}^{-1}$.
 \end{theorem}
 \begin{proof}
 	By matrix similarity, we have 
 	\begin{align*}
 	\sigma({\bf M}^{-1}{\bf A}_0)=\sigma({\bf M}^{-\frac{1}{2}}{\bf A}_0{\bf M}^{-\frac{1}{2}})=\sigma({\bf M}^{-\frac{1}{2}}({\bf M}+\tau{\bf K}){\bf M}^{-\frac{1}{2}})=\sigma({\bf I}_{J}+\tau{\bf M}^{-\frac{1}{2}}{\bf K}{\bf M}^{-\frac{1}{2}}),
 	\end{align*}
 	which implies that $\sigma({\bf M}^{-1}{\bf A}_0)\in(1,+\infty)$. Thus, $\sigma({\bf A}_0^{-1}{\bf M})=\sigma(({\bf M}^{-1}{\bf A}_0)^{-1})\in(0,1)$. By $\epsilon\in(0,1]$, we know that $\sigma(\epsilon({\bf A}_0^{-1}{\bf M})^{N})\in(0,1)$. That means $0\notin\sigma({\bf I}_{J}-\epsilon({\bf A}_0^{-1}{\bf M})^{N})$, which proves that ${\bf Z}_{\epsilon}$ is invertible.
 	
 	It is clear that ${\bf P}_{\epsilon}$ can be rewritten as ${\bf P}_{\epsilon}={\bf L}-\epsilon{\bf E}_1{\bf M}{\bf E}_{N}^{\rm T}$. Using this expression of ${\bf P}_{\epsilon}$, it is straightforward to verify that ${\bf P}_{\epsilon}({\bf L}^{-1}+{\bf L}^{-1}{\bf E}_1{\bf Z}_{\epsilon}^{-1}{\bf E}_{N}^{\rm T}{\bf L}^{-1})={\bf I}_{NJ}$, which shows that ${\bf P}_{\epsilon}$ is invertible and ${\bf P}_{\epsilon}^{-1}={\bf L}^{-1}+{\bf L}^{-1}{\bf E}_1{\bf Z}_{\epsilon}^{-1}{\bf E}_{N}^{\rm T}{\bf L}^{-1}$.
 \end{proof}
 
 \begin{rem}
 	As shown in Theorem \ref{pinvertibility}, $\epsilon\in(0,1]$ guarantees the invertibility of ${\bf P}_{\epsilon}$. Hence, throughout this paper, we choose $\epsilon\in(0,1]$.
 \end{rem}
 
 With BEC preconditioner, instead of solving \eqref{allatoncesystem}, we employ Krylov subspace methods to solve the preconditioned system as follows
 \begin{equation}\label{preconditionedsystem}
 {\bf P}_{\epsilon}^{-1}{\bf L}{\bf u}={\bf P}_{\epsilon}^{-1}{\bf f}.
 \end{equation}
 
 \begin{theorem}\label{preconditionedspectrum}
 	\begin{description}
 		\item[(i)]The preconditioned matrix ${\bf P}_{\epsilon}^{-1}{\bf L}$ has a identity-plus-low-rank decomposition, i.e., ${\rm rank}({\bf P}_{\epsilon}^{-1}{\bf L}-{\bf I}_{NJ})= J$. Hence, ${\bf P}_{\epsilon}^{-1}{\bf L}$ has exactly $(N-1)J$ many eigenvalues equal to 1.
 		\item[(ii)] Given any constant $\eta\in(0,1)$, take $\epsilon\in(0,\eta]$. Then, $\max\limits_{\lambda\in\sigma({\bf P}_{\epsilon}^{-1}{\bf L})}|\lambda-1|\leq\frac{\epsilon}{1-\eta}$.
 	\end{description}
 \end{theorem}
 \begin{proof}
 	By Theorem \ref{pinvertibility}, ${\bf P}_{\epsilon}^{-1}{\bf L}-{\bf I}_{NJ}={\bf L}^{-1}{\bf E}_1{\bf Z}_{\epsilon}^{-1}{\bf E}_{N}^{\rm T}$, with ${\bf Z}_{\epsilon}$ defined in \eqref{zepsdef}. Then, $${\rm rank}({\bf L}^{-1}{\bf E}_1{\bf Z}_{\epsilon}^{-1}{\bf E}_{N}^{\rm T})={\rm rank}({\bf E}_1{\bf Z}_{\epsilon}^{-1}{\bf E}_{N}^{\rm T})= J,$$
 	which proves ${\bf (i)}$.
 	
 	Substituting \eqref{elinvlttexpres} into ${\bf P}_{\epsilon}^{-1}{\bf L}={\bf I}_{NJ}+{\bf L}^{-1}{\bf E}_1{\bf Z}_{\epsilon}^{-1}{\bf E}_{N}^{\rm T}$, we obtain
 	\begin{equation}\label{preconditionedmatblockexpan}
 	{\bf P}_{\epsilon}^{-1}{\bf L}=\left[\begin{array}[c]{cccc}
 	{\bf I}_{J}&&&({\bf L}^{-1})_{0}{\bf Z}_{\epsilon}^{-1}\\
 	&{\bf I}_{J}&&({\bf L}^{-1})_{1}{\bf Z}_{\epsilon}^{-1}\\
 	&&\ddots&\vdots\\
 	&&&{\bf I}_{J}+({\bf L}^{-1})_{N-1}{\bf Z}_{\epsilon}^{-1}
 	\end{array}\right]
 	\end{equation}
 	Therefore, $\sigma({\bf P}_{\epsilon}^{-1}{\bf L})=\{1\}\cup\sigma({\bf I}_{J}+({\bf L}^{-1})_{N-1}{\bf Z}_{\epsilon}^{-1})$. And then,
 	\begin{equation*}
 	\max\limits_{\lambda\in\sigma({\bf P}_{\epsilon}^{-1}{\bf L})}|\lambda-1|=\max\limits_{\lambda\in\sigma({\bf I}_{J}+({\bf L}^{-1})_{N-1}{\bf Z}_{\epsilon}^{-1})}|\lambda-1|.
 	\end{equation*}
 	It thus remains to investigate $\sigma({\bf I}_{J}+({\bf L}^{-1})_{N-1}{\bf Z}_{\epsilon}^{-1})$.
 	By \eqref{elinvlttexpres} and definition of ${\bf Z}_{\epsilon}$ given in Theorem \ref{pinvertibility}, 
 	\begin{equation}\label{lastdiagblockexpans}
 	{\bf I}_{J}+({\bf L}^{-1})_{N-1}{\bf Z}_{\epsilon}^{-1}={\bf I}_{J}+\epsilon[({\bf M}^{-1}{\bf A}_0)^{N}-\epsilon{\bf I}_{J}]^{-1},
 	\end{equation}
 	which implies that
 	\begin{align*}
 	\sigma({\bf I}_{J}+({\bf L}^{-1})_{N-1}{\bf Z}_{\epsilon}^{-1})&=\left\{1+\epsilon(\lambda^{N}-\epsilon)^{-1}|\lambda\in\sigma({\bf M}^{-1}{\bf A}_0)\right\}\\
 	&=\left\{\frac{\lambda^N}{\lambda^N-\epsilon}\bigg|\lambda\in\sigma({\bf M}^{-\frac{1}{2}}{\bf A}_0{\bf M}^{-\frac{1}{2}})\right\}\\
 	&=\left\{\frac{\lambda^N}{\lambda^N-\epsilon}\bigg|\lambda\in\sigma({\bf I}_{J}+\tau{\bf M}^{-\frac{1}{2}}{\bf K}{\bf M}^{-\frac{1}{2}})\right\}\subset\left\{\frac{\lambda^N}{\lambda^N-\epsilon}\bigg|\lambda\in(1,+\infty)\right\}.
 	\end{align*}
 	Hence,
 	\begin{align*}
 	\max\limits_{\lambda\in\sigma({\bf I}_{J}+({\bf L}^{-1})_{N-1}{\bf Z}_{\epsilon}^{-1})}|\lambda-1|&\leq\sup\limits_{\lambda\in(1,+\infty)}\left|\frac{\lambda^{N}}{\lambda^N-\epsilon}-1\right|\\
 	&=\sup\limits_{\lambda\in(1,+\infty)}\left|\frac{\epsilon}{\lambda^N-\epsilon}\right|\leq \frac{\epsilon}{1-\eta},
 	\end{align*}
 	which completes the proof.
 \end{proof}
 
 Theorem \ref{preconditionedspectrum}${\bf (i)}$ implies that by using GMRES method, the exact solution of the preconditioned system \eqref{preconditionedsystem} can be found within at most $J+1$ iterations. But this is not a sharp estimation of convergence rate of GMRES method when $J$ is not small. In Theorem \ref{gmrescvgthm}, we will show that GMRES method for the system \eqref{preconditionedsystem} has a linear convergence rate independent of $N$ and $J$ when $\epsilon\lesssim\sqrt{\tau}$. Theorem \ref{preconditionedspectrum}${\bf (ii)}$ shows that all the eigenvalues of the preconditioned matrix are clustered at 1 with clustering radius of $\mathcal{O}(\epsilon)$.
 
 \begin{lemma}\label{diagonalizableblock}
 	There exists an invertible matrix ${\bf V}\in\mathbb{R}^{J\times J}$ and a diagonal matrix ${\bf D}\in\mathbb{R}^{J\times J}$ such that
 	$ {\bf I}_{J}+({\bf L}^{-1})_{N-1}{\bf Z}_{\epsilon}^{-1}={\bf V}{\bf D}{\bf V}^{-1}$ and $1\notin\sigma({\bf D})$, where ${\bf Z}_{\epsilon}$ is defined in \eqref{zepsdef}.
 \end{lemma}
 \begin{proof}
 	Denote ${\bf H}_0:={\bf I}_{J}+\epsilon[({\bf M}^{-\frac{1}{2}}{\bf A}_0{\bf M}^{-\frac{1}{2}})^N-\epsilon{\bf I}_J]^{-1}$.
 	From \eqref{lastdiagblockexpans}, we know that
 	\begin{align*}
 	{\bf I}_{J}+({\bf L}^{-1})_{N-1}{\bf Z}_{\epsilon}^{-1}={\bf I}_{J}+\epsilon[({\bf M}^{-1}{\bf A}_0)^{N}-\epsilon{\bf I}_{J}]^{-1}={\bf M}^{-\frac{1}{2}}{\bf H}_0{\bf M}^{\frac{1}{2}}.
 	\end{align*}
 	Since ${\bf M}^{-\frac{1}{2}}{\bf A}_0{\bf M}^{-\frac{1}{2}}$ is real symmetric, so is ${\bf H}_0$. Thus, ${\bf H}_0$ is orthogonally diagonalizable, i.e, there exists an orthogonal matrix ${\bf Q}\in\mathbb{R}^{J\times J}$ and a diagonal matrix ${\bf D}\in\mathbb{R}^{J\times J}$ such that ${\bf H}_0={\bf Q}{\bf D}{\bf Q}^{\rm T}$. Letting ${\bf V}={\bf M}^{-\frac{1}{2}}{\bf Q}$, we then obtain ${\bf I}_{J}+({\bf L}^{-1})_{N-1}{\bf Z}_{\epsilon}^{-1}={\bf V}{\bf D}{\bf V}^{-1}={\bf V}{\bf D}{\bf V}^{-1}$.
 	
 	By ${\bf H}_0={\bf Q}{\bf D}{\bf Q}^{\rm T}$, definition of ${\bf H}_0$ and $\epsilon\in (0,1]$, we know that
 	\begin{align*}
 	\sigma({\bf D})=\sigma({\bf H}_0)=\left\{\frac{\lambda^N}{\lambda^N-\epsilon}\bigg|\lambda\in\sigma({\bf M}^{-\frac{1}{2}}{\bf A}_0{\bf M}^{-\frac{1}{2}})\right\}&=\left\{\frac{\lambda^N}{\lambda^N-\epsilon}\bigg|\lambda\in\sigma({\bf I}_J+\tau{\bf M}^{-\frac{1}{2}}{\bf K}{\bf M}^{-\frac{1}{2}})\right\}\\
 	&\subset(1,+\infty),
 	\end{align*}
 	which means $1\notin\sigma({\bf D})$.
 \end{proof}
 \begin{theorem}\label{preconditionedmatdiagonalization}
 	The preconditioned matrix ${\bf P}^{-1}{\bf L}$ is diagonalizable, i.e.,
 	\begin{equation*}
 	{\bf P}_{\epsilon}^{-1}{\bf L}=\hat{\bf V}\hat{\bf D}\hat{\bf V}^{-1},
 	\end{equation*}
 	where
 	\begin{align*}
 	&\hat{\bf V}=\left[\begin{array}[c]{ccccc}
 	{\bf I}_{J}&~&~&~&{\bf V}_0\\
 	~&{\bf I}_{J}&~&~&{\bf V}_1\\
 	~&~&\ddots&~&\vdots\\
 	~&~&~&{\bf I}_{J}&{\bf V}_{N-2}\\
 	~&~&~&~&-{\bf V}
 	\end{array}\right],\qquad \hat{\bf D}=\left[\begin{array}[c]{ccccc}
 	{\bf I}_{J}&~&~&~&~\\
 	~&{\bf I}_{J}&~&~&~\\
 	~&~&\ddots&~&~\\
 	~&~&~&{\bf I}_{J}&~\\
 	~&~&~&~&{\bf D}
 	\end{array}\right],\\
 	&{\bf V}_i=({\bf L}^{-1})_{i}{\bf Z}_{\epsilon}^{-1}{\bf V}({\bf I}_{J}-{\bf D})^{-1},\quad i=0,1,...,N-2,
 	\end{align*}
 	with ${\bf V}$ and ${\bf D}$ given by Lemma \ref{diagonalizableblock}.
 \end{theorem}
 \begin{proof}
 	By Lemma \ref{diagonalizableblock}, $1\notin\sigma({\bf D})$, i.e., ${\bf I}_{J}-{\bf D}$ is invertible. Thus, ${\bf V}_{i}~(i=0,1,...,N-2)$ are well-defined. Then, it is straightforward to verify that ${\bf L}\hat{\bf V}={\bf P}_{\epsilon}\hat{\bf V}\hat{\bf D}$. Moreover, invertibility of ${\bf V}$ guarantees the invertibility of $\hat{\bf V}$. That means ${\bf P}_{\epsilon}^{-1}{\bf L}=\hat{\bf V}\hat{\bf D}\hat{\bf V}^{-1}$. The proof is complete.
 \end{proof}

 Let ${\bf O}$ denote zero matrix with proper size.
 
 \begin{lemma}\label{preconditionedmatapproximateidlem}
 	Given any $\eta\in(0,1)$, choose $\epsilon\in(0,\eta]$. Then,
 	\begin{equation*}
 	||{\bf P}_{\epsilon}^{-1}{\bf L}-{\bf I}_{NJ}||_2\leq\frac{\epsilon c_0\sqrt{N}}{1-\eta},
 	\end{equation*}
 	where $c_0:=\sup\limits_{J\in\mathbb{N}^{+}}\kappa_2({\bf M}^{\frac{1}{2}})=\sqrt{\sup\limits_{J\in\mathbb{N}^{+}}\kappa_2({\bf M})}<+\infty$ is independent of $J$ and $N$.
 \end{lemma}
 \begin{proof}	
 	As ${\bf M}^{\frac{1}{2}}{\bf A}_0^{-1}{\bf M}^{\frac{1}{2}}$ is real symmetric, ${\bf M}^{\frac{1}{2}}{\bf A}_0^{-1}{\bf M}^{\frac{1}{2}}$ is orthogonally diagonalizable, i.e., there exists an orthogonal matrix ${\bf Q}\in\mathbb{R}^{J\times J}$ and a diagonal matrix ${\bf \Lambda}\in\mathbb{R}^{J\times J}$ such that ${\bf M}^{\frac{1}{2}}{\bf A}_0^{-1}{\bf M}^{\frac{1}{2}}={\bf Q}{\bf\Lambda}{\bf Q}^{\rm T}$. Since $\sigma({\bf \Lambda})=\sigma({\bf M}^{\frac{1}{2}}{\bf A}_0^{-1}{\bf M}^{\frac{1}{2}})$, ${\bf M}^{\frac{1}{2}}{\bf A}_0^{-1}{\bf M}^{\frac{1}{2}}=[{\bf I}_J+\tau{\bf M}^{-\frac{1}{2}}{\bf K}{\bf M}^{-\frac{1}{2}}]^{-1}$ implies that ${\bf O}\prec{\bf \Lambda}\preceq{\bf I}_{J}$.
 	Then, by \eqref{elinvlttexpres} and definition of ${\bf Z}_{\epsilon}$ given in \eqref{zepsdef}, we have
 	\begin{align}
 	({\bf L}^{-1})_k{\bf Z}_{\epsilon}^{-1}&=\epsilon({\bf A}_0^{-1}{\bf M})^{k+1}[{\bf I}_{J}-\epsilon({\bf A}_0^{-1}{\bf M})^N]^{-1}\notag\\
 	&=\epsilon{\bf M}^{-\frac{1}{2}}({\bf M}^{\frac{1}{2}}{\bf A}_0^{-1}{\bf M}^{\frac{1}{2}})^{k+1}[{\bf I}_{J}-\epsilon({\bf M}^{\frac{1}{2}}{\bf A}_0^{-1}{\bf M}^{\frac{1}{2}})^{N}]^{-1}{\bf M}^{\frac{1}{2}}\notag\\
 	&=\epsilon{\bf M}^{-\frac{1}{2}}{\bf Q}{\bf\Lambda}^{k+1}[{\bf I}_{J}-\epsilon{\bf\Lambda}^{N}]^{-1}{\bf Q}{\bf M}^{\frac{1}{2}},\quad k=0,1,...,N-1,\notag
 	\end{align}
 	which together with \eqref{preconditionedmatblockexpan} implies that
 	\begin{align*}
 	{\bf P}_{\epsilon}^{-1}{\bf L}-{\bf I}_{NJ}&=\left[\begin{array}[c]{cccc}
 	~&&&({\bf L}^{-1})_{0}{\bf Z}_{\epsilon}^{-1}\\
 	&~&&({\bf L}^{-1})_{1}{\bf Z}_{\epsilon}^{-1}\\
 	&&~&\vdots\\
 	&&&({\bf L}^{-1})_{N-1}{\bf Z}_{\epsilon}^{-1}
 	\end{array}\right]\\
 	&=\epsilon[{\bf I}_{N}\otimes({\bf M}^{-\frac{1}{2}}{\bf Q})]\left[\begin{array}[c]{cccc}
 	~&&&{\bf\Lambda}^{1}[{\bf I}_{J}-\epsilon{\bf\Lambda}^{N}]^{-1}\\
 	&~&&{\bf\Lambda}^{2}[{\bf I}_{J}-\epsilon{\bf\Lambda}^{N}]^{-1}\\
 	&&~&\vdots\\
 	&&&{\bf\Lambda}^{N}[{\bf I}_{J}-\epsilon{\bf\Lambda}^{N}]^{-1}
 	\end{array}\right][{\bf I}_{N}\otimes({\bf Q}^{\rm T}{\bf M}^{\frac{1}{2}})].
 	\end{align*}
 	Rewrite ${\bf\Lambda}={\rm diag}(\lambda_i)_{i=1}^{J}$. Then,
 	\begin{align*}
 	||{\bf P}_{\epsilon}^{-1}{\bf L}-{\bf I}_{NJ}||_2&\leq \epsilon ||{\bf I}_{N}\otimes({\bf M}^{-\frac{1}{2}}{\bf Q})||_{2}||{\bf I}_{N}\otimes({\bf Q}^{\rm T}{\bf M}^{\frac{1}{2}})||_2\sqrt{\left|\left|\sum\limits_{k=1}^{N}{\bf \Lambda}^{2k}({\bf I}_J-\epsilon{\bf\Lambda}^{N})^{-2}\right|\right|_2}\\
 	&=\epsilon\kappa_2({\bf M}^{\frac{1}{2}})\sqrt{\max\limits_{1\leq i\leq J}\sum\limits_{k=1}^{N}\left(\frac{\lambda_i^k}{1-\epsilon\lambda_i^N}\right)^2}\leq\epsilon c_0\sqrt{\max\limits_{1\leq i\leq J}\sum\limits_{k=1}^{N}\left(\frac{\lambda_i^k}{1-\epsilon\lambda_i^N}\right)^2}.
 	\end{align*}
 	Moreover, it is easy to check that the functions $g_k(x):=\frac{x^k}{1-\epsilon x^N}$ are monotonically increasing on $x\in[0,1]$ for each $k=1,2,...,N$. Since ${\bf O}\prec{\bf \Lambda}\preceq{\bf I}_{N}$, $\{\lambda_i|1\leq i\leq J\}\subset[0,1]$.
 	Hence,
 	\begin{equation*}
 	||{\bf P}_{\epsilon}^{-1}{\bf L}-{\bf I}_{NJ}||_2\leq \epsilon c_0\sqrt{\sum\limits_{k=1}^{N}\frac{1}{(1-\epsilon)^2}}=\frac{\epsilon c_0\sqrt{N}}{1-\epsilon}\leq \frac{\epsilon c_0\sqrt{N}}{1-\eta},
 	\end{equation*}
 	which completes the proof
 \end{proof}
 
 For any matrix ${\bf Z}\in\mathbb{R}^{m\times m}$, denote
 \begin{equation*}
 \mathcal{H}({\bf Z}):=\frac{{\bf Z}+{\bf Z}^{\rm T}}{2},\quad \mathcal{S}({\bf Z}):=\frac{{\bf Z}-{\bf Z}^{\rm T}}{2}.
 \end{equation*}
 
 Let $\lambda_{\min}(\cdot)$ and $\lambda_{\max}(\cdot)$ denote the minimal and maximal eigenvalue of a Hermitian matrix, respectively. Let $\rho(\cdot)$ denotes the spectral radius of a square matrix.
 \begin{lemma}\textnormal{(see \cite[(1.1)]{beckermann2005some})}\label{gmrescvglm}
 	Let ${\bf \Xi}{\bf q}={\bf w}$ be a real square linear system with $\mathcal{H}({\bf \Xi})\succ{\bf O}$. Then, the residuals of the iterates generated by applying GMRES to solving ${\bf \Xi}{\bf v}={\bf w}$ satisfy
 	\begin{equation*}
 	||{\bf r}_k||_2\leq\left(1-\frac{\lambda_{\min}(\mathcal{H}({\bf \Xi}))^2}{||{\bf \Xi}||_2^2}\right)^{k/2}||{\bf r}_0||_2,
 	\end{equation*}
 	where ${\bf r}_k={\bf w}-{\bf \Xi}{\bf q}_k$ with ${\bf q}_k$ $(k\geq 1)$ being the iterate solution at $k$th GMRES iteration and ${\bf q}_0$ being an arbitrary initial guess.
 \end{lemma}
 \begin{theorem}\label{gmrescvgthm}
 	For any given constants $\delta\in(0,1)$, choose $\epsilon\in(0,b_{\tau}]$, where\\ $b_{\tau}:=\frac{\delta\sqrt{\tau}}{\delta\sqrt{\tau}+c_0\sqrt{T}}$ and $c_0$ is given by Lemma \ref{preconditionedmatapproximateidlem}.
 	Then, the residuals of the iterates generated by applying  GMRES to solving the preconditioned system \eqref{preconditionedsystem} satisfy
 	\begin{equation*}
 	||{\bf r}_k||_2\leq\left(\frac{2\sqrt{\delta}}{1+\delta}\right)^{k}||{\bf r}_0||_2,
 	\end{equation*}
 	where ${\bf r}_k={\bf P}_{\epsilon}^{-1}{\bf f}-{\bf P}_{\epsilon}^{-1}{\bf L}{\bf u}_k$ with ${\bf u}_k~(k\geq 1)$ being the iterative solution at $k$th GMRES iteration and ${\bf u}_0$ denoting an arbitrary initial guess.
 \end{theorem}
 \begin{proof}
 	Denote ${\bf\Xi}={\bf P}_{\epsilon}^{-1}{\bf L}-{\bf I}_{NJ}$.
 	Since $b_{\tau}\in(0,1)$, Lemma \ref{preconditionedmatapproximateidlem} is applicable. By Lemma \ref{preconditionedmatapproximateidlem}, we have
 	\begin{equation*}
 	||{\bf\Xi}||_2\leq \frac{\epsilon c_0\sqrt{N}}{1-b_{\tau}}=\frac{\epsilon\delta}{b_{\tau}}\leq\delta.
 	\end{equation*}
 	Then,
 	\begin{equation}\label{hpartpositive}
 	\mathcal{H}({\bf P}_{\epsilon}^{-1}{\bf L})={\bf I}_{NJ}+\mathcal{H}({\bf\Xi})\succeq (1-\delta){\bf I}_{NJ}\succ{\bf O},
 	\end{equation}
 	implies that Lemma \ref{gmrescvglm} is applicable to the preconditioned system \eqref{preconditionedsystem}. It remains to estimate $\lambda_{\min}(\mathcal{H}({\bf P}_{\epsilon}^{-1}{\bf L}))^2$ and $||{\bf P}_{\epsilon}^{-1}{\bf L}||_{2}^2$.
 	
 	Clearly, \eqref{hpartpositive} implies that
 	\begin{equation*}
 	\lambda_{\min}(\mathcal{H}({\bf P}_{\epsilon}^{-1}{\bf L}))^2\geq (1-\delta)^2.
 	\end{equation*}
 	Moreover,
 	\begin{equation*}
 	||{\bf P}_{\epsilon}^{-1}{\bf L}||_{2}^2=||({\bf P}_{\epsilon}^{-1}{\bf L})^{\rm T}{\bf P}_{\epsilon}^{-1}{\bf L}||_2=||{\bf I}_{NJ}+{\bf \Xi}+{\bf \Xi}^{\rm T}+{\bf \Xi}^{\rm T}{\bf\Xi}||_2\leq (1+2\delta+\delta^2)=(1+\delta)^2.
 	\end{equation*}
 	Then, Lemma \ref{gmrescvglm} implies that
 	\begin{align*}
 	||{\bf r}_{k}||_2&\leq\left(1-\frac{\lambda_{\min}(\mathcal{H}({\bf P}_{\epsilon}^{-1}{\bf L}))^2}{||{\bf P}_{\epsilon}^{-1}{\bf L}||_{2}^2}\right)||{\bf r}_0||_2\\
 	&\leq\left(1-\frac{(1-\delta)^2}{(1+\delta)^2}\right)^{k/2}||{\bf r}_0||_2=\left(\frac{2\sqrt{\delta}}{1+\delta}\right)^k||{\bf r}_0||_2,
 	\end{align*}
 	which completes the proof.
 \end{proof}
 
 \begin{rem}
 	Theorem \ref{gmrescvgthm} shows that GMRES for the preconditioned system \eqref{preconditionedsystem} has a linear convergence rate independent of system size whenever $0<\epsilon\lesssim\sqrt{\tau}$.  Actually, as illustrated by numerical results in Section 5, taking $\epsilon=\mathcal{O}(\tau)$ already leads to a fast convergence of GMRES.
 \end{rem}
 \section{Implementation}\label{implementation}
 In this section, we discuss on how to efficiently implement the GMRES method for the preconditioned system \eqref{preconditionedsystem}. In GMRES iteration, it requires to compute the matrix-vector product, ${\bf P}_{\epsilon}^{-1}({\bf L}{\bf v})$ for some given vector ${\bf v}$. 
 In this section, we present a fast implementation for computing the matrix-vector product.
 Since our presented fast implementation also works when $\partial_t$ is discretized by multi-step backward difference, we start with multi-step-backward-difference discretization of $\partial_t$ to describe the fast implementation.

 Discretizing $\partial_t$ by a $p$-step backward difference scheme, then the corresponding ${\bf L}$ is as follows \cite{mcdonald2018preconditioning}
 \begin{equation}\label{pstepscheme}
 {\bf L}={\bf R}\otimes{\bf M}+\tau{\bf I}_{N}\otimes{\bf K},
 \end{equation}
 where `$\otimes$' denotes the Kronecker product,
 \begin{equation*}
 {\bf R}:=\left[\begin{array}[c]{cccccc}
 r_0&&&&&\\
 r_1&r_0&&&&\\
 \vdots&\ddots&\ddots&&&\\
 r_p&\ddots&\ddots&\ddots&&\\
 &\ddots&\ddots&r_1&r_0&\\
 &&r_p&\ldots&r_1&r_0
 \end{array}\right]\in\mathbb{R}^{N\times N}
 \end{equation*}
 Note that if $p=1$, $r_0=1$, $r_1=-1$, then the $p$-step \eqref{pstepscheme} scheme reduces to the backward difference scheme presented in Section \ref{discretization}. For $p$-step scheme, the corresponding BEC preconditioner ${\bf P}_{\epsilon}$ is defined as follows
 \begin{equation}\label{pstepbecpreconditioner}
 {\bf P}_{\epsilon}={\bf R}_{\epsilon}\otimes{\bf M}+\tau{\bf I}_{N}\otimes{\bf K},
 \end{equation}
 where
 \begin{equation*}
 {\bf R}_{\epsilon}=\left[\begin{array}[c]{cccccc}
 r_0&&\epsilon r_{p}&\ldots&\epsilon r_2&\epsilon r_1\\
 r_1&r_0&&\ddots&\ddots&\epsilon r_2\\
 \vdots&\ddots&\ddots&&\ddots&\vdots\\
 r_p&\ddots&\ddots&\ddots&&\epsilon r_p\\
 &\ddots&\ddots&r_1&r_0&\\
 &&r_p&\ldots&r_1&r_0
 \end{array}\right]\in\mathbb{R}^{N\times N}.
 \end{equation*}

 For a given vector ${\bf v}\in\mathbb{R}^{NJ\times 1}$, to compute ${\bf P}_{\epsilon}^{-1}{\bf L}{\bf v}$ is equivalent to compute $\tilde{\bf v}={\bf L}{\bf v}$ and ${\bf P}_{\epsilon}^{-1}\tilde{\bf v}$. Hence, to compute the preconditioned-matrix-vector product efficiently, it suffices to compute both ${\bf P}_{\epsilon}^{-1}{\bf v}$ and ${\bf L}{\bf v}$ efficiently for a given vector ${\bf v}$.
 
 \begin{proposition}\textnormal{(see \cite[(2)]{van2000ubiquitous})}\label{krnckprodmatvec}
 	For any ${\bf B}\in\mathbb{C}^{p_1\times q_1},~{\bf C}\in\mathbb{C}^{p_2\times q_2},~{\bf Y}=({\bf y}_1,{\bf y}_2,...,{\bf y}_{q_1})\in\mathbb{C}^{q_2\times q_1}$, it holds $({\bf B}\otimes{\bf C})({\bf y}_1;{\bf y}_2;...;{\bf y}_{q_1})=({\bf C}{\bf Y}{\bf B}^{\rm T})(:)=[({\bf B}({\bf C}{\bf Y})^{\rm T})^{\rm T}](:)$.
 \end{proposition}
Proposition \cite{van2000ubiquitous} shows that if two matrices ${\bf B}$ and ${\bf C}$ have fast matrix-vector product, then their Kronecker product ${\bf B}\otimes{\bf C}$ also have fast matrix-vector product. More specifically, if computing  ${\bf B}\in\mathbb{C}^{p\times p}$ ($ {\bf C}\in\mathbb{C}^{q\times q}$, respectively) times a vector requires operations no more than $c_1$ ($c_2$, respectively), then computing ${\bf B}\otimes{\bf C}$ times a vector requires operations no more than $c_1q+c_2p$. 
\begin{definition}
A square matrix ${\bf G}\in\mathbb{C}^{m\times m}$ is called a Toeplitz if and only if it has the form of
\begin{equation*}
	{\bf G}=\left[
	\begin{array}
		[c]{ccccc}
		g_0 & g_{-1} &\ldots   &g_{2-m}  &g_{1-m}\\
		g_1 & g_0&g_{-1}   &\ldots &g_{2-m}\\
		\vdots&\ddots &\ddots&\ddots&\vdots\\
		g_{m-2} & \ldots&g_1& g_0 &g_{-1}\\
		g_{m-1} & g_{m-2}& \ldots & g_1 & g_{0}
	\end{array}
	\right].
\end{equation*}
If additionally  $g_{i-j}=g_{i-j-m}$ for all $i-j\geq 1$, then ${\bf G}$ is called a circulant matrix.
\end{definition}

Denote
\begin{equation}\label{fouriermatdef}
	{\bf F}_{m}=\frac{1}{\sqrt{m}}\left[\theta_m^{(i-1)(j-1)}\right]_{i,j=1}^{m},\quad \theta_m=\exp\left(\frac{2\pi{\bf i}}{m}\right),\quad {\bf i}=\sqrt{-1}.
\end{equation}
 ${\bf F}_m$ is called a Fourier transform matrix. ${\bf F}_{m}^{*}$ (${\bf F}_{m}$, respectively) times a vector is equivalent to Fourier transform (inverse Fourier transform, respectively) of the vector up to a scaling constant. Hence,  ${\bf F}_{m}^{*}$ (or ${\bf F}_{m}$) times a vector can be fast computed by algorithms of fast Fourier transform (FFT), which requires $\mathcal{O}(m\log m)$ operations and $\mathcal{O}(m)$ storage.
 
 For a vector ${\bf v}$, by ${\rm diag}({\bf v})$, we denote the diagonal matrix with entries of ${\bf v}$ as its diagonal elements. It is well known that any circulant matrix ${\bf C}\in\mathbb{C}^{m\times m}$ is diagonalizable by Fourier transform matrix (see, e.g., \cite{ng2004iterative,chan2007introduction}):
 \begin{equation*}
{\bf C}={\bf F}_n{\rm diag}(\sqrt{m}{\bf F}_m^{*}{\bf C}(:,1)){\bf F}_m^{*},
 \end{equation*}
where ${\bf C}(:,1)$ denotes the first column of ${\bf C}$. Hence, computing an $m\times m$ circulant matrix times a vector requires $\mathcal{O}(N\log N)$ operations and $\mathcal{O}(N)$ storage by FFTs. Any Toeplitz matrix ${\bf G}\in\mathbb{C}^{m\times m}$ can be embedded into a larger circulant matrix (see, e.g., \cite{ng2004iterative,chan2007introduction,chan1996conjugate}): 
\begin{equation*}
\left[
\begin{array}[c]{cc}
	{\bf G}&\times\\
	\times&{\bf G}
\end{array}
\right],
\end{equation*}
 where ``$\times$" here denotes some proper blocks. Hence, an $m\times m$ Toeplitz matrix times a vector ${\bf G}{\bf v}$ can be computed as
 \begin{equation*}
\left[
\begin{array}[c]{cc}
	{\bf G}&\times\\
	\times&{\bf G}
\end{array}
\right]\left[\begin{array}[c]{c}
{\bf v}\\
{\bf 0}
\end{array}\right]=\left[\begin{array}[c]{c}
{\bf G}{\bf v}\\
\times 
\end{array}\right],
 \end{equation*}
 which requires $\mathcal{O}(m\log m)$ operations and $\mathcal{O}(m)$ storage.
 
 We firstly discuss the fast computation of ${\bf L}{\bf v}$ for a given vector ${\bf v}$. Note that ${\bf M}$, ${\bf I}_N$ and ${\bf K}$ are all sparse matrices. Moreover, for small $p$, ${\bf R}$ is sparse and thus ${\bf L}$ is sparse. It is well-known that computing a sparse matrix times a vector requires a linear complexity. In other words, when $p$ is small, the computation of ${\bf L}{\bf v}$ requires $\mathcal{O}(NJ)$ storage and operations. Notice also that ${\bf R}$ is a Toeplitz matrix no matter how large $p$ is. As ${\bf L}$ consists of Kronecker products of Toeplitz matrix and sparse sparse matrix, sparse matrix and sparse matrix, Proposition \ref{krnckprodmatvec} implies that the computation of ${\bf L}{\bf v}$ requires $\mathcal{O}(JN\log N)$ operations and $\mathcal{O}(JN)$ storage by FFTs no matter how big $p$ is.

 Now, we focus on the fast computation of ${\bf P}_{\epsilon}^{-1}{\bf y}$ for a given vector $y\in\mathbb{R}^{NJ\times 1}$. To this end, we exploit an interesting property of the matrix ${\bf R}_{\epsilon}$, i.e., its diagonalizable property. From \cite[Theorem 2.10]{bini2005numerical}, we know that ${\bf R}_{\epsilon}$ can be diagonalized as follows
 \begin{equation}\label{repsdiagform}
 {\bf R}_{\epsilon}={\bf D}_{\epsilon}^{-1}{\bf F}_{N}^{*}{\bf\Lambda}_{\epsilon}{\bf F}_{N}{\bf D}_{\epsilon},
 \end{equation}
 where ${\bf F}_N$ is defined in \eqref{fouriermatdef},
 \begin{align*}
 &{\bf D}_{\epsilon}={\rm diag}\left(\epsilon^{\frac{0}{N}},\epsilon^{\frac{1}{N}},...,\epsilon^{\frac{N-1}{N}}\right),\qquad {\bf \Lambda}_{\epsilon}={\rm diag}(\lambda_{0}^{(\epsilon)},\lambda_{1}^{(\epsilon)},...,\lambda_{N-1}^{(\epsilon)}),\\
 &\lambda_{k}^{(\epsilon)}=\sum\limits_{j=0}^{p}r_j\epsilon^{\frac{j}{N}}\theta_N^{kj},\quad k=0,1,...,N-1,
 \end{align*}
$\theta_N$ is defined in \eqref{fouriermatdef}.
 
 When $p$ is small, then it is clear that the computation of $\{\lambda_{k}^{(\epsilon)}\}_{k=0}^{N-1}$ requires $\mathcal{O}(N)$ operations and storage. When $p$ is large, one can exploit the fact that
 \begin{equation*}
 (\lambda_{0}^{(\epsilon)},\lambda_{1}^{(\epsilon)},...,\lambda_{N-1}^{(\epsilon)})^{\rm T}=\sqrt{N}{\bf F}_N(r_0\epsilon^{\frac{0}{N}},r_1\epsilon^{\frac{1}{N}},...,r_p\epsilon^{\frac{p}{N}},0,0,...,0)^{\rm T}.
 \end{equation*} 
 Hence, using IFFT, the computation of $\{\lambda_{k}^{(\epsilon)}\}_{k=0}^{N-1}$ requires $\mathcal{O}(N\log N)$ operations and $\mathcal{O}(N)$ storage, no matter how big $p$ is.
 
 By \eqref{repsdiagform}, ${\bf P}_{\epsilon}$ can be rewritten as the following block diagonalization form
 \begin{equation}\label{becblockdiag}
 {\bf P}_{\epsilon}=[({\bf D}_{\epsilon}^{-1}{\bf F}_{N}^{*})\otimes{\bf I}_J]{\rm blockdiag}({\bf B}_0,{\bf B}_1,...,{\bf B}_{N-1})[({\bf F}_{N}{\bf D}_{\epsilon})\otimes{\bf I}_J],
 \end{equation}
 where
 \begin{equation*}
 {\bf B}_k=\lambda_{k}^{(\epsilon)}{\bf M}+\tau{\bf K},\quad k=0,1,...,N-1.
 \end{equation*}
 Let ${\bf y}=({\bf y}^1;{\bf y}^2;\cdots;{\bf y}^N)\in\mathbb{R}^{NJ\times 1}$ with ${\bf y}^k\in\mathbb{R}^{J\times 1}~(k=1,2,...,N)$ be a given vector. Then, the computation of ${\bf z}={\bf P}_{\epsilon}^{-1}{\bf y}$ can be equivalently rewritten as the following 3 steps:
 \begin{align}
 {\rm Step~1}: \quad & {\rm Compute~}\tilde{{\bf y}}=\left[({\bf F}_N {\bf D}_{\epsilon})\otimes {\bf I}_{J} \right]{\bf y},\label{ytild}\\
 {\rm Step~2}: \quad & {\rm Solve~} {\bf B}_{k-1}\tilde{\bf z}^k=\tilde{\bf y}^k{\rm ~for~}\tilde{\bf z}^k,~ k=1,2,...,N,{\rm ~where~}\left(\tilde{\bf y}^1;\tilde{\bf y}^2;\cdots;\tilde{\bf y}^N\right)=\tilde{{\bf y}},\label{blockdiaginv}\\
 {\rm Step~3}: \quad & {\rm Compute~}{\bf z}=\left[({\bf D}_{\epsilon}^{-1}{\bf F}_N^{\ast})\otimes {\bf I}_{J} \right]\tilde{\bf z},{\rm ~where~}\tilde{\bf z}=\left(\tilde{\bf z}^1;\tilde{\bf z}^2;\cdots;\tilde{\bf z}^N\right).\label{zcompt}
 \end{align}
 Using FFTs and Proposition \ref{krnckprodmatvec}, it is easy to see that \eqref{ytild} and \eqref{zcompt} requires $\mathcal{O}(JN\log N)$ operations and $\mathcal{O}(JN)$ storage. If the spatial discretization is finite difference method or finite element method with uniform square grid and the diffusion coefficient function $a$ is a constant, then ${\bf B}_{k}'s~(k=0,1,...,N-1)$ are all diagonalizable by means of fast sine transform (see \cite{mcdonald2018preconditioning}), in the case of which the $N$ many linear systems in \eqref{blockdiaginv} can be fast and directly solved with $\mathcal{O}(NJ\log J)$ operations and $\mathcal{O}(NJ)$ storage. In a more general situation that ${\bf B}_{k}'s$ are not diagonalizable, one can use some efficient spatial solvers, such as a multigrid method to solve the linear systems in \eqref{blockdiaginv}, for which only a few iterations are required since ${\bf P}_{\epsilon}$ serves as a preconditioner. Although the linear systems in \eqref{blockdiaginv} are complex, numerical results in Section 5 show that one iteration of V-cycle geometric multigrid for solving each linear system in \eqref{blockdiaginv} already leads to a fast convergence of GMRES for the preconditioned system. Solving the linear systems in \eqref{blockdiaginv} by V-cycle multigrid method with a fixed number of iterations, it requires $\mathcal{O}(NJ)$ operations and storage.
 
 It is remarkable to note that only half of the $N$ many systems in \eqref{blockdiaginv} need to be solved, the reason of which is explained as follows. From \eqref{ytild}, we know that the right hand sides in \eqref{blockdiaginv} can be expressed as
 \begin{align*}
 \tilde{\bf y}^{k+1}=\frac{1}{\sqrt{N}}\sum\limits_{j=0}^{p}\epsilon^{\frac{j}{N}}\theta_{N}^{kj}{\bf y}^{j+1},\quad k=0,1,...,N-1.
 \end{align*}
 Recall that the matrices in \eqref{blockdiaginv} have the following expressions
 \begin{equation*}
 {\bf B}_{k}=\left(\sum\limits_{j=0}^{p}r_j\epsilon^{\frac{j}{N}}\theta_{N}^{kj}\right){\bf M}+\tau{\bf K},\quad k=0,1,...,N-1.
 \end{equation*}
 
 Let ${\rm conj}(\cdot)$ denote conjugate of a matrix or a vector.
 Then,
 \begin{align*}
 &{\rm conj}(\tilde{\bf y}^{k+1})=\frac{1}{\sqrt{N}}\sum\limits_{j=0}^{p}\epsilon^{\frac{j}{N}}\theta_{N}^{-kj}{\bf y}^{j+1}=\frac{1}{\sqrt{N}}\sum_{j=0}^{p}\epsilon^{\frac{j}{N}}\theta_{N}^{(N-k)j}{\bf y}^{j+1}=\tilde{\bf y}^{N-k+1},\quad 1\leq k\leq N-1,\\
 &{\rm conj}({\bf B}_{k})=\left(\sum\limits_{j=0}^{p}r_j\epsilon^{\frac{j}{N}}\theta_{N}^{-kj}\right){\bf M}+\tau{\bf K}=\left(\sum\limits_{j=0}^{p}r_j\epsilon^{\frac{j}{N}}\theta_{N}^{(N-k)j}\right){\bf M}+\tau{\bf K}={\bf B}_{N-k},~ 1\leq k\leq N-1.
 \end{align*}
 That means the unknowns in \eqref{blockdiaginv} hold equalities: $\tilde{\bf z}^{k+1}={\rm conj}(\tilde{\bf z}^{N-k+1})$ for $k=1,2,...,N-1$. Hence, only the first $\left\lceil\frac{N+1}{2}\right\rceil$ many linear systems in \eqref{blockdiaginv} need to be solved.
 
 Hence, when ${\bf M}$ and ${\bf K}$ are diagonalizable by the fast sine transform, the computation of ${\bf P}_{\epsilon}^{-1}{\bf L}{\bf v}$ for a given vector ${\bf v}$ can be fast and exactly implemented, which requires $\mathcal{O}(NJ)$ storage and $\mathcal{O}(NJ\log J)$ operations. In other more general cases, the computation of ${\bf P}_{\epsilon}^{-1}{\bf L}{\bf v}$ for a given vector ${\bf v}$ can be approximately implemented by V-cycle multigrid method with a fixed number of iterations, which requires $\mathcal{O}(NJ)$ operations and storage. Hence, using the multigrid method, the computation of ${\bf P}_{\epsilon}^{-1}{\bf L}{\bf v}$ for a given vector ${\bf v}$ requires $\mathcal{O}(NJ)$ storage and $\mathcal{O}(NJ)$ operations. 
 
 From the above discussion, we see that each preconditioned GMRES iteration requires even cheaper operations by using the multigrid inner solver than that by using the fast sine transform solver. However, unlike multigrid method, the fast sine transform solver is an exact solver for \eqref{blockdiaginv} that does not bring additional iterative error.  Hence, when the fast sine transform solver is applicable, we prefer to use the fast sine transform solver.
 \section{Numerical Results}
 In this section, we test the performance of the proposed BEC preconditioner through examples of heat equation, convection diffusion equation and compare it with block circulant preconditioner proposed in \cite{mcdonald2018preconditioning}. Finite element discretization with $Q1$ element and uniform square mesh is used to discretize the spatial terms of all the examples in this section. In Examples \ref{varcoeffheateq} and \ref{cdeq}, the mass matrix ${\bf M}$ and the stiffness matrix ${\bf K}$ are generated by the IFISS package \cite{ifiss}.
 All numerical experiments are performed via MATLAB R2016a on a workstation equipped with dual Xeon E5-2690 v4 14-Cores 2.6GHz CPUs,  256GB RAM running CentOS Linux version 7.
 
 Restarted GMRES method is employed to solve the preconditioned systems. The restarting number of GMRES is set as $50$. The tolerance of GMRES is set as $||{\bf r}_k||_2\leq 10^{-7}||{\bf r}_0||_2$, where ${\bf r}_k$ denotes preconditioned residual vector at $k$th GMRES iteration. The zero vector is used as initial guess of GMRES method.

 For convenience, the block circulant preconditioner is denoted by BC. As the BC preconditioner is a special case of BEC preconditioner. Hence, we use the same algorithm for implementation of BC preconditioner as the one used for that of BEC preconditioner. We also denote GMRES with BC and BEC preconditioners by GMRES-BC and GMRES-BEC, respectively.
 
 Since preconditioned residual error by GMRES-BEC has a different definition from that by GMRES-BC, we define the following unpreconditioned relative residual error to measure the accuracy of GMRES-BEC and GMRES-BC for fair comparison:
 \begin{equation*}
 {\rm RES}:=\frac{||{\bf f}-{\bf L}{\bf u}_{\rm iter}||_2}{||{\bf f}||_2},
 \end{equation*}
where ${\bf u}_{\rm iter}$ denotes some iterative solution.
 
By `Iter', we denote the iteration number of restarted GMRES; by `DoF', the number of degrees of freedom, i.e., the number of unknowns, and by `CPU', the computational time in seconds. 
 
 For all the numerical experiments in this section, we take $\epsilon=\min\{0.5,0.5\tau\}$ for the BEC preconditioner.
 
 \begin{expl}\label{constheateq}
 	{\rm 
 		The first example is heat equation \eqref{heateq}--\eqref{inicd} with
 		\begin{equation*}
 		\Omega=(0,1)\times(0,1),\quad T=1,\quad f\equiv 0,\quad a\equiv 10^{-5},\quad g\equiv 0,\quad u_0=x(x-1)y(y-1).
 		\end{equation*}
 		For Example \ref{constheateq}, the corresponding ${\bf B}_{k}'s$ in \eqref{blockdiaginv} is diagonalizable by sine transform. Hence, we implement the matrix-vector multiplication by fast sine transform for Example \ref{constheateq}. To demonstrate that the proposed preconditioning method works for multi-step temporal discretization scheme, we firstly discretize the temporal derivative of Example \ref{constheateq} by the two-step backward difference scheme (BDF2). The BDF2 scheme is defined by $r_0=\frac{3}{2}$, $r_1=-2$, $r_2=\frac{1}{2}$ and $p=2$ (see the meanings of $r_k$'s and $p$ in \eqref{pstepscheme}). The results of GMRES-BEC and GMRES-BC for all-at-once system from BDF2 temporal scheme is listed in Table \ref{bdf2table}. Table \ref{bdf2table} shows that (i) both GMRES-BEC and GMRES-BC work for BDF2-type all-at-once system; (ii) GMRES-BEC is more efficient than GMRES-BC in terms of computational time and iteration number; (iii) GMRES-BEC is more accurate than GMRES-BC in terms of {\rm RES} measure.
 		 
 			\begin{table}[H]
 			\begin{center}
 				\caption{Performance of GMRES-BC and GMRES-BEC on Example \ref{constheateq} discretized by BDF2 scheme}\label{bdf2table}
 				\setlength{\tabcolsep}{0.8em}
 				\begin{tabular}[c]{ccc|ccc|ccc}
 					\hline
 					~&~&~& \multicolumn{3}{c|}{GMRES-BEC}& \multicolumn{3}{c}{GMRES-BC}\\
 					\cline{4-9}
 					$N$&$J+1$&DoF&Iter&CPU&{\rm RES}&Iter&CPU&{\rm RES}\\
 					\hline
 					\multirow{5}{*}{$2^{6}$}
 					&$2^6$   &254016   &13&1.40 &9.98e-7   &82&4.06  &6.00e-3 \\
 					&$2^7$   &1032256  &13&2.82 &9.98e-7   &80&12.53 &5.10e-3 \\
 					&$2^8$   &4161600  &13&12.81&9.98e-7   &79&71.35 &5.20e-3 \\
 					&$2^{9}$ &16711744 &13&53.12&9.98e-7   &80&299.86&5.00e-3 \\
 					\hline
 						\multirow{5}{*}{$2^{7}$}
 					&$2^6$   &508032   &13&1.36  &1.01e-6   &80&6.51  &9.20e-3 \\
 					&$2^7$   &2064512  &13&5.21  &1.01e-6   &77&28.78 &9.00e-3 \\
 					&$2^8$   &8323200  &13&27.97 &1.01e-6   &77&158.63&8.60e-3 \\
 					&$2^{9}$ &33423488 &13&105.35&1.01e-6   &76&567.45&8.90e-3 \\
 					\hline
 						\multirow{5}{*}{$2^{8}$}
 					&$2^6$   &1016064  &13&2.27  &1.01e-6   &71&11.18  &1.42e-2 \\
 					&$2^7$   &4129024  &13&12.44 &1.01e-6   &70&60.81  &1.39e-2 \\
 					&$2^8$   &16646400 &13&52.75 &1.01e-6   &67&259.32 &1.32e-2 \\
 					&$2^{9}$ &66846976 &13&204.88&1.01e-6   &68&1013.56&1.28e-2 \\
 					\hline
 					\multirow{5}{*}{$2^{9}$}
 				    &$2^6$   &2032128   &12&4.84  &3.03e-6   &65&25.53  &1.77e-2 \\
 				    &$2^7$   &8258048   &12&25.00 &3.03e-6   &64&127.27 &1.69e-2 \\
 				    &$2^8$   &33292800  &12&100.23&3.03e-6   &61&480.22 &1.57e-2 \\
 				    &$2^{9}$ &133693952 &12&385.97&3.03e-6   &60&1796.61&1.57e-2 \\ 
 					\hline	
 				\end{tabular}
 			\end{center}
 		\end{table}	
 	
 	In the rest of this section, we focus on testing the performance of GMRES-BC and GMRES-BEC for all-at-once system from 1-step backward difference scheme\eqref{backwarddiff}. By BDF, we denote the 1-step backward difference scheme \eqref{backwarddiff}. 
 		The results of GMRES-BEC and GMRES-BC preconditioner for solving Example \ref{constheateq} discretized by BDF are listed in Tables \ref{table1}.
 		
 		Table \ref{table1} shows that (i) GMRES-BEC is more efficient than GMRES-BC in terms of CPU and iteration number; (ii) the convergence rates of both GMRES-BEC and GMRES-BC are independent of temporal and spatial stepsizes; (iii) GMRES-BEC is more accurate than GMRES-BC in terms of {\rm RES}.
 		
 		\begin{table}[H]
 			\begin{center}
 				\caption{Performance of GMRES-BC and GMRES-BEC on Example \ref{constheateq} discretized by BDF scheme}\label{table1}
 				\setlength{\tabcolsep}{0.8em}
 					\begin{tabular}[c]{ccc|ccc|ccc}
 					\hline
 					~&~&~& \multicolumn{3}{c|}{GMRES-BEC}& \multicolumn{3}{c}{GMRES-BC}\\
 					\cline{4-9}
 					$N$&$J+1$&DoF&Iter&CPU&{\rm RES}&Iter&CPU&{\rm RES}\\
 					\hline
 					\multirow{5}{*}{$2^{6}$}
 					&$2^6$   &254016   &2 &0.55 &9.11e-11   &13&1.21 &2.09e-5 \\
 					&$2^7$   &1032256  &2 &0.73 &1.69e-10   &13&2.71 &2.80e-5 \\
 					&$2^8$   &4161600  &2 &3.08 &2.23e-10   &13&13.81&3.08e-5 \\
 					&$2^{9}$ &16711744 &2 &12.65&2.46e-10   &13&53.27&3.16e-5 \\
 					\hline
 					\multirow{5}{*}{$2^{7}$}
 					&$2^6$   &508032   &2 &0.31  &2.27e-11  &13&1.18 &2.09e-5 \\
 					&$2^7$   &2064512  &2 &1.23  &4.21e-11  &13&4.96 &2.81e-5 \\
 					&$2^8$   &8323200  &2 &6.88  &5.56e-11  &12&24.05&3.44e-5 \\
 					&$2^{9}$ &33423488 &2 &25.04 &6.16e-11  &13&11.92&2.81e-5 \\
 					\hline
 					\multirow{5}{*}{$2^{8}$}
 					&$2^6$   &1016064  &2 &0.54  &5.69e-12  &13&2.16  &2.09e-5 \\
 					&$2^7$   &4129024  &2 &3.01  &1.05e-11  &13&11.92 &2.81e-5 \\
 					&$2^8$   &16646400 &2 &12.54 &1.60e-11  &13&53.46 &3.08e-5 \\
 					&$2^{9}$ &66846976 &2 &49.89 &1.55e-11  &13&203.58&3.16e-5 \\
 					\hline
 					\multirow{5}{*}{$2^{9}$}
 					&$2^6$   &2032128   &1 &0.91  &5.87e-8   &13&4.91  &2.09e-5 \\
 					&$2^7$   &8258048   &1 &4.69  &5.99e-8   &13&28.86 &2.81e-5 \\
 					&$2^8$   &33292800  &1 &18.78 &6.03e-8   &13&104.66&3.08e-5 \\
 					&$2^{9}$ &133693952 &1 &74.01 &6.05e-8   &13&406.46&3.16e-5 \\ 
 					\hline	
 				\end{tabular}
 			\end{center}
 		\end{table}

 	}
 \end{expl}
 \begin{expl}\label{varcoeffheateq}
 	{\rm The second example is also a heat equation but with variable diffusion coefficient function $a$, which is defined as follows
 		\begin{align*}
 		\Omega=&(0,1)\times(0,1),\quad T=1,\quad a(x,y)=10^{-5}\times\sin(\pi xy),\quad g\equiv 0,\quad u_0=x(x-1)y(y-1),\\
 		f(x,y,t)	=&\exp(-t)x(1-x)[2\sin(\pi xy)-y(1-y)-\pi\cos(\pi xy)x(1-2y)]+\\
 		&\exp(-t)y(1-y)[2\sin(\pi xy)-\pi\cos(\pi xy)y(1-2x)].
 		\end{align*}
 		Example \ref{varcoeffheateq} has the closed form analytical solution as follows
 		\begin{equation*}
 		u(x,y,t)=\exp(-t)x(1-x)y(1-y).
 		\end{equation*}
 		Hence, for Example \ref{varcoeffheateq}, we can measure the error of its numerical solution. For this purpose, we define the error function as follows
 		\begin{equation*}
 		{\rm E}_{N,J}=||{\bf u}_{\rm iter}-{\bf u}^{*}||_{\infty},
 		\end{equation*}
 		where ${\bf u}_{\rm iter}$ denotes the iterative solution of the linear system \eqref{allatoncesystem}, ${\bf u}^{*}$ denotes the values of exact solution of the heat equation on the mesh. Since the exact solution of Example \ref{varcoeffheateq} is known, instead of ${\rm RES}$, we use ${\rm E}_{N,J}$ to measure the accuracy of GMRES-BC and GMRES-BEC. The temporal derivative in Example \ref{varcoeffheateq} is discretized by the BDF \eqref{allatoncesystem}. Notice that ${\bf B}_{k}'s$ in \eqref{blockdiaginv} arising from Example \ref{varcoeffheateq} is no longer diagonalizable by sine transform. Hence, for Example \ref{varcoeffheateq}, instead of solving \eqref{blockdiaginv} exactly, we approximately solve it by one iteration of V-cycle geometric multigrid method, in which ILU smoother is employed with one time of pre-smoothing and one time of post-smoothing; the piecewise linear interpolation and its transpose are used as the interpolation and restriction operators (see \cite{ifiss}). The results of GMRES-BEC and GMRES-BC for solving Example \ref{varcoeffheateq} are listed in Table \ref{table2}.
 		
 		From Table \ref{table2} shows that (i) GMRES-BEC is much more efficient than GMRES-BC in terms of CPU and iteration number;  (ii) the iteration number of GMRE-BEC keeps bounded as $N$ and $J$ changes. That means introducing the parameter $\epsilon$ indeed help improve the performance of BC preconditioner on Example \ref{varcoeffheateq}. 
 		\captionsetup[subfigure]{labelformat=empty}
 		\begin{table}[H]
 			\begin{center}
 				\caption{Performance of GMRES-BC and GMRES-BEC on Example \ref{varcoeffheateq}}\label{table2}
 				\setlength{\tabcolsep}{0.7em}
 				\begin{tabular}[c]{ccc|ccc|ccc}
 					\hline
 					~&~&~& \multicolumn{3}{c|}{GMRES-BEC}& \multicolumn{3}{c}{GMRES-BC}\\
 					\cline{4-9}
 					$N$&$J+1$&DoF&Iter&CPU&${\rm E}_{N,J}$&Iter&CPU&${\rm E}_{N,J}$\\
 					\hline
 					\multirow{4}{*}{$2^{6}$}
 					&$2^6$   &254016  &3&1.29 &2.95e-4&72 &9.07   &2.95e-4\\
 					&$2^7$   &1032256 &3&2.20 &3.05e-4&78 &34.10  &3.04e-4\\
 					&$2^8$   &4161600 &2&7.47 &3.07e-4&87 &162.58 &3.07e-4\\
 					&$2^{9}$ &16711744&2&44.00&3.08e-4&133&1124.33&3.08e-4\\
 					\hline
 					\multirow{4}{*}{$2^{7}$}
 					&$2^6$   &508032  &3&0.93 &1.41e-4  &72 &16.98  &1.42e-4\\
 					&$2^7$   &2064512 &3&3.40 &1.51e-4  &78 &65.01  &1.51e-4\\
 					&$2^8$   &8323200 &2&13.51&1.53e-4  &87 &328.69 &1.53e-4\\
 					&$2^{9}$ &33423488&2&64.77&1.54e-4  &133&2214.82&1.54e-4\\         
 					\hline
 					\multirow{4}{*}{$2^{8}$}
 					&$2^6$   &1016064  &3&1.78  &6.43e-5  &72 &30.96  &6.50e-5\\
 					&$2^7$   &4129024  &3&6.60  &7.39e-5  &78 &126.79 &7.39e-5\\
 					&$2^8$   &16646400 &2&23.04 &7.63e-5  &87 &638.34 &7.96e-5\\
 					&$2^{9}$ &66846976 &2&115.18&7.69e-5  &133&4479.13&8.38e-5\\ 
 					\hline
 					\multirow{4}{*}{$2^{9}$}
 					&$2^6$   &2032128  &3&3.53  &2.57e-5  &72 &60.84  &2.65e-5\\
 					&$2^7$   &8258048  &3&13.45 &3.54e-5  &78 &260.44 &3.55e-5\\
 					&$2^8$   &33292800 &2&45.24 &3.78e-5  &87 &1251.80&7.97e-5\\
 					&$2^{9}$ &133693952&2&217.20&3.84e-5  &133&8996.60&8.39e-5\\	                                             
 					\hline	
 				\end{tabular}
 			\end{center}
 		\end{table}	
 		
 		To visualize the numerical solution of Example \ref{varcoeffheateq}, we present its surface plot and contour plot in Figure \ref{varheatsolfig}.
 		\begin{figure}[H]
 			\centering
 			\subfloat[Surface plot]{\includegraphics[width=2.2in]{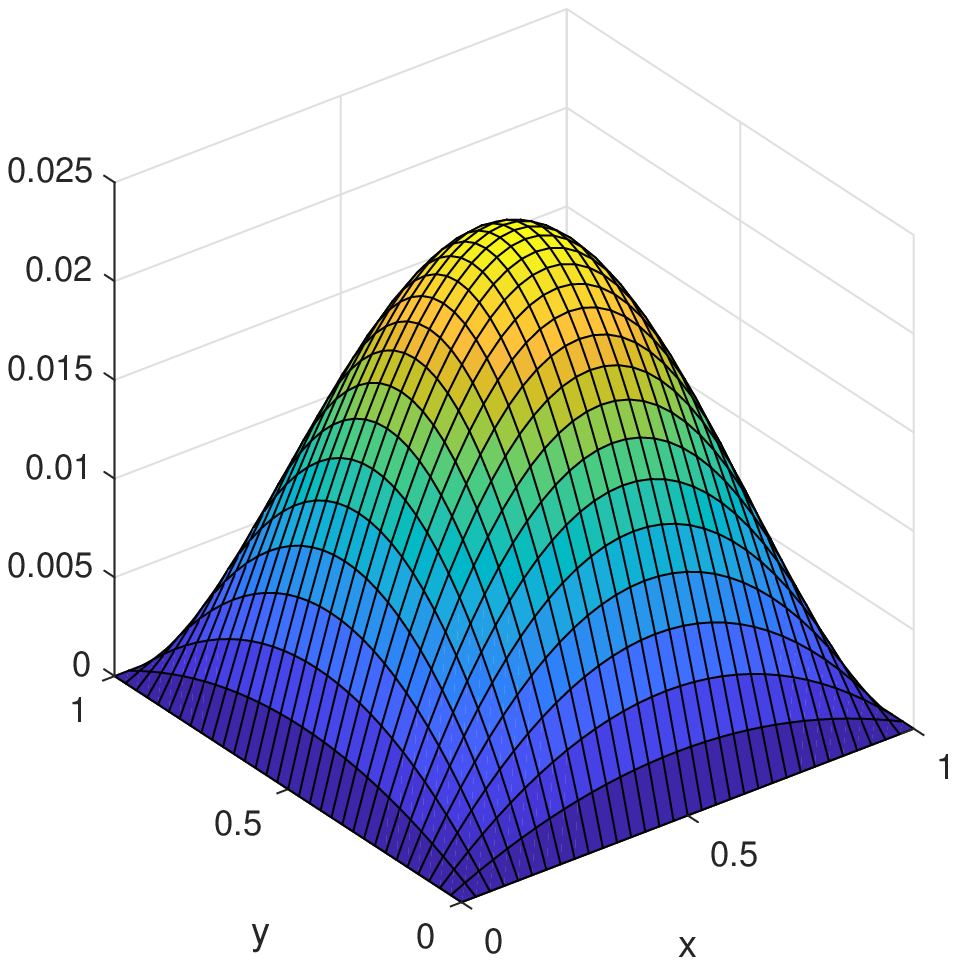}}\qquad
 			\subfloat[Contour plot]{\includegraphics[width=2.2in]{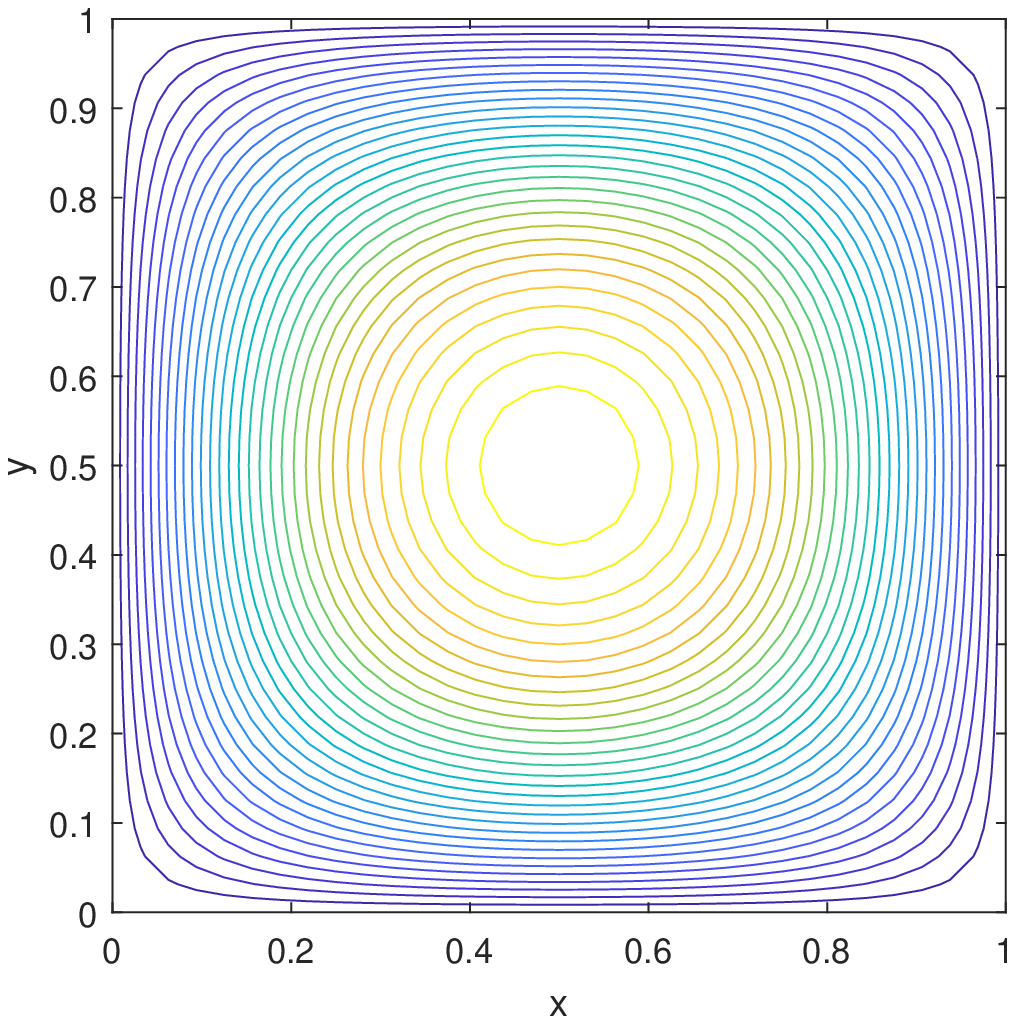}}
 			\caption{Numerical solution of Example \ref{varcoeffheateq} at final time $T$ by GMRES-BEC with $N=20$ and $J=31$}\label{varheatsolfig}
 		\end{figure}
 	}
 \end{expl}
 
 \begin{expl}\textnormal{(see \cite{mcdonald2018preconditioning})}\label{cdeq}
 	{\rm 
 		The third example is an evolutionary convection diffusion equation with circulating wind and hot wall boundary, which is defined as follows
 		\begin{align*}
 		&\partial_t u(x,y,t)=\frac{1}{200}\Delta u-\overrightarrow{w}\cdot\nabla u, (x,y)\in\Omega:=(-1,1)\times (-1,1),\quad t\in(0,T],\\
 		&u(x,y,t)=(1-\exp(-10t))\phi(x,y),\quad (x,y)\in\partial\Omega,\\
 		&u(x,y,0)=0,\quad (x,y)\in\bar{\Omega},
 		\end{align*}
 		where $\overrightarrow{w}:=(2y(1-x^2),-2x(1-y^2))$ is the circulating wind, $\phi$ represents the hot wall boundary condition defined as follows
 		\begin{equation*}
 		\phi(x,y):=\begin{cases}
 		1,\quad x=1{\rm ~and~} (x,y)\in\partial\Omega,\\
 		0,\quad x\neq 1{\rm ~and~} (x,y)\in\partial\Omega.
 		\end{cases}
 		\end{equation*}
 		The steady-state version of Example \ref{cdeq} is given by \cite[Example 6.1.4]{elman2014finite}. The Streamline-upwind Petrov-Galerkin (SUPG) stabilization \cite{brooks1982streamline} is used to stabilize the discrete spatial terms. The temporal derivative in Example \ref{cdeq} is discretized by the BDF scheme \eqref{backwarddiff}.
 		We solve \eqref{blockdiaginv} arising from Example \ref{cdeq} by one iteration of V-cycle geometric multigrid method, in which ILU smoother is employed with one time of pre-smoothing and one time of post-smoothing; the piecewise linear interpolation and its transpose are used as the interpolation and restriction operators (see \cite{ifiss}). The results of GMRES-BEC and GMRES-BC for solving Example \ref{cdeq} are listed in Table \ref{table3}.
 		
 		Table \ref{table3} shows that (i) GMRES-BEC is more efficient than GMRES-BC on Example \ref{cdeq} in terms of CPU and iteration number; (ii) GMRES-BEC is more accurate than GMRES-BC in terms of {\rm RES}.

 			\begin{table}[H]
 			\begin{center}
 				\caption{Performance of GMRES-BC and GMRES-BEC on Example \ref{cdeq} with $T=1$}\label{table3}
 				\setlength{\tabcolsep}{0.8em}
 				\begin{tabular}[c]{ccc|ccc|ccc}
 					\hline
 					~&~&~& \multicolumn{3}{c|}{GMRES-BEC}& \multicolumn{3}{c}{GMRES-BC}\\
 					\cline{4-9}
 					$N$&$J+1$&DoF&Iter&CPU&{\rm RES}&Iter&CPU&{\rm RES}\\
 					\hline
 					\multirow{5}{*}{$2^{6}$}
 					&$2^6$   &254016   &5 &1.67 &6.51e-8   &20&2.72  &3.72e-7 \\
 					&$2^7$   &1032256  &5 &2.95 &1.44e-8   &21&9.60  &8.62e-8 \\
 					&$2^8$   &4161600  &5 &12.84&9.05e-9   &21&41.32 &4.44e-8 \\
 					&$2^{9}$ &16711744 &5 &66.79&3.07e-9   &21&189.99&1.70e-8 \\
 					\hline
 					\multirow{5}{*}{$2^{7}$}
 					&$2^6$   &508032   &5 &1.49  &3.43e-8  &21&4.87  &2.93e-7 \\
 					&$2^7$   &2064512  &5 &5.07  &1.16e-8  &21&17.73 &1.87e-7 \\
 					&$2^8$   &8323200  &5 &24.32 &9.57e-9  &22&83.85 &4.00e-8 \\
 					&$2^{9}$ &33423488 &5 &112.30&3.47e-9  &22&381.04&1.42e-8 \\
 					\hline
 					\multirow{5}{*}{$2^{8}$}
 					&$2^6$   &1016064  &5 &2.75  &1.79e-8   &21&9.59  &4.83e-7 \\
 					&$2^7$   &4129024  &5 &10.10 &1.10e-8   &22&37.98 &1.48e-7 \\
 					&$2^8$   &16646400 &5 &45.13 &1.09e-8   &22&160.84&7.56e-8 \\
 					&$2^{9}$ &66846976 &5 &214.81&3.86e-9   &22&752.86&2.67e-8 \\
 					\hline
 					\multirow{5}{*}{$2^{9}$}
 					&$2^6$   &2032128   &4 &5.25  &1.75e-7   &21&19.27  &7.41e-7 \\
 					&$2^7$   &8258048   &5 &19.62 &1.20e-8   &22&78.47  &2.37e-7 \\
 					&$2^8$   &33292800  &5 &85.36 &1.31e-8   &22&334.48 &1.24e-7 \\
 					&$2^{9}$ &133693952 &5 &404.25&4.56e-9   &22&1469.73&4.37e-8 \\ 
 					\hline	
 				\end{tabular}
 			\end{center}
 		\end{table}

 		Since the boundary condition of Example \ref{cdeq} converges to the steady state, one can expect that solution of Example \ref{cdeq} will be very close to its steady-state solution for sufficiently large $T$. To observe this, we present the numerical solution of Example \ref{cdeq} at $T=200$ by GMRES-BEC in Figure \ref{cdeqsolfig}. Indeed, the numerical solution exhibited in Figure \ref{cdeqsolfig} is very closed to the numerical steady-state solution exhibited in \cite[FIG. 6.5]{elman2014finite}.
 		\captionsetup[subfigure]{labelformat=empty}
 		\begin{figure}[H]
 			\centering
 			\subfloat[Surface plot]{\includegraphics[width=2.2in]{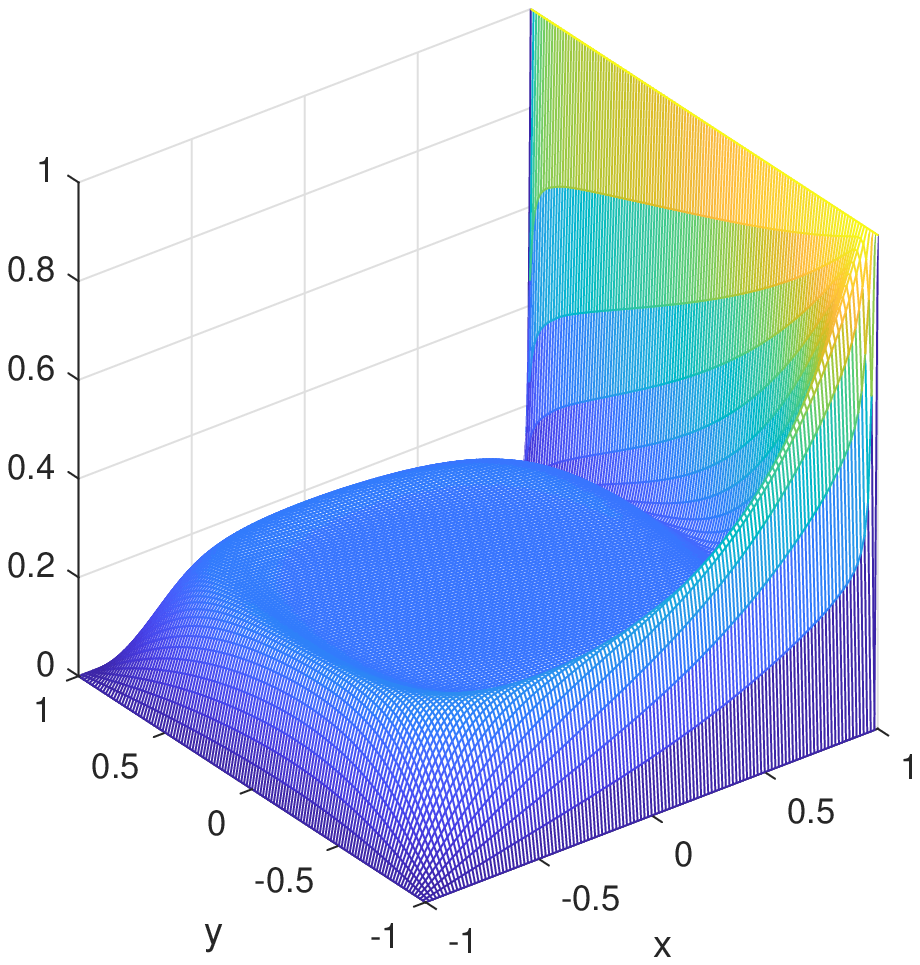}}\qquad
 			\subfloat[Contour plot]{\includegraphics[width=2.2in]{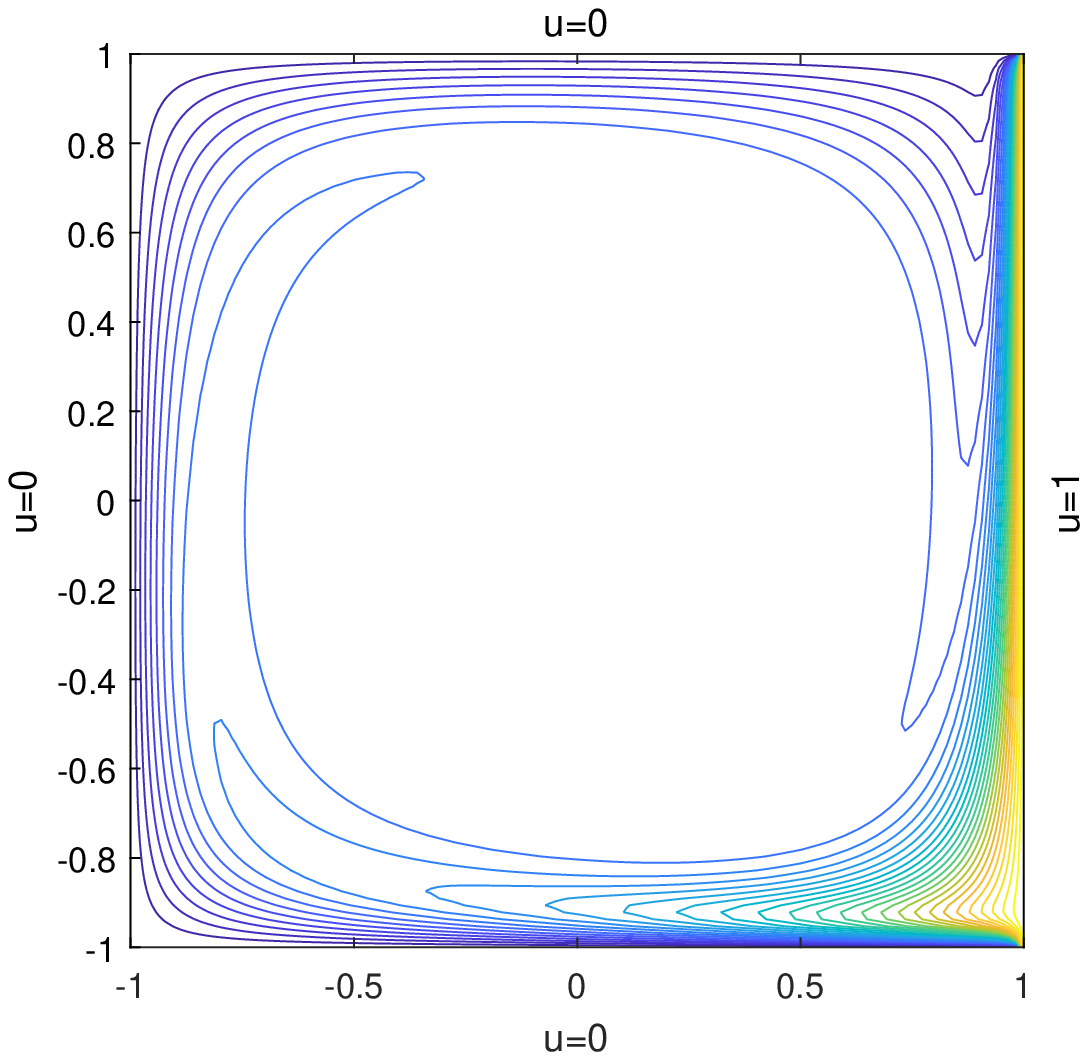}}
 			\caption{Numerical solution of Example \ref{cdeq} at time $T=200$ by GMRES-BEC with $N=200$ and $J=127$}\label{cdeqsolfig}
 		\end{figure}
 	}
 \end{expl}
 
 \section{Concluding Remark}
 In this paper, we have proposed the BEC preconditioner as a generalization of BC preconditioner for all-at-once system arising from evolutionary PDEs by introducing a positive parameter $\epsilon$ into the top-right corner of BC preconditioner. We have shown that such generalization preserves the diagonalizability, identity-plus-low-rank decomposition of the preconditioned matrix. Moreover, when $\epsilon$ is sufficiently small, we have shown that (i) the preconditioned matrix by BEC preconditioner has all eigenvalues clustered at $1$; (ii) GMRES for the preconditioned system by BEC preconditioner has a linear convergence rate independent of matrix-size. A fast implementation has been introduced so that the computational complexity required for implementation of BEC preconditioner stays the same as that for BC preconditioner. Numerical results have shown that BEC preconditioner improves the performance of the BC preconditioner.
 
\bibliographystyle{siam}
\bibliography{myreference}
\end{document}